\documentclass[a4paper,12pt]{article}
\usepackage{blindtext}

\usepackage{subfiles} 
\usepackage[pagewise]{lineno}
\usepackage{cite,titlesec,graphics,amsmath, amssymb, amscd, amsthm, euscript, hyperref,
indentfirst,fancyhdr,pgfplots,stmaryrd, textcomp, enumerate, mathrsfs
} \usetikzlibrary{arrows.meta,calc,bending,decorations.markings,decorations.pathmorphing}

\graphicspath{{.}{./eps/}}

\hypersetup{
    colorlinks=true,
    linkcolor=blue,
    filecolor=magenta,      
    urlcolor=cyan,
    pdftitle={Affine-Additive},
    pdfpagemode=FullScreen,
    }
\makeatletter \theoremstyle{plain}

\newcommand{\R}{\mathbb R}
\newcommand{\N}{\mathbb N}

\newcommand{\Aa}{\mathcal{AA}}

\renewcommand{\L}{\mathcal L}
\newcommand{\C}{\mathbb C}

\renewcommand{\H}{\mathbb H}
\newcommand{\bH}{\mathbf H}
\newcommand{\bG}{\mathbb G}

\newcommand{\RT}{\mathcal{RT}}
\newcommand{\Mod}{\mathrm{Mod}}

\newtheorem{thm}{Theorem}[section]
\newtheorem{cor}[thm]{Corollary}
\newtheorem{prop}[thm]{Proposition}

\theoremstyle{definition}
\newtheorem{defn}[thm]{Definition}
\newtheorem{exam}[thm]{Example}

\DeclareMathOperator{\dist}{\text{dist}}
\DeclareMathOperator{\diam}{\text{diam}}

\makeatother
\pgfplotsset{compat=1.18}

\newcommand\blfootnote[1]{%
  \begingroup
  \renewcommand\thefootnote{}\footnote{#1}%
  \addtocounter{footnote}{-1}%
  \endgroup
}

\providecommand{\keywords}[1]
{
  \small	
  {\textit{Key words and phrases.}} #1
}

\providecommand{\subjclass}[1]
{
  \small	
  {\textit{$2020$ Mathematics Subject Classifications.}} #1
}

\begin{document}

\title{Hyperbolicity of the sub-Riemannian affine-additive group}

\author{Zolt\'an M. Balogh
  \and
  Elia  Bubani
  \and
  Ioannis D. Platis
}

\newcommand{\Addresses}{{
  \bigskip
  \footnotesize

  Z.M.~Balogh, (Corresponding author), \textsc{Universit\"at Bern,
Mathematisches Institut (MAI), \\
Sidlerstrasse 5,
3012 Bern,
Switzerland.}\par\nopagebreak
  \textit{E-mail address}, Z.M.~Balogh: \texttt{zoltan.balogh@unibe.ch}

  \medskip

  E.~Bubani,  \textsc{Universit\"at Bern,
Mathematisches Institut (MAI), 
Sidlerstrasse 5,
3012 Bern,
Switzerland.}\par\nopagebreak
  \textit{E-mail address}, E.~Bubani: \texttt{elia.bubani@unibe.ch}

  \medskip

  I.D.~Platis, \textsc{Department of Mathematics and Applied Mathematics, University of Crete
Voutes Campus, 70013 Heraklion, Greece}\par\nopagebreak
  \textit{E-mail address}, I.D.~Platis: \texttt{jplatis@math.uoc.gr}

}}
\date{}

\maketitle
\centerline{\textit{Dedicated to Hans Martin Reimann}}

\begin{abstract}
We consider the affine-additive group as a metric measure space with a canonical left-invariant measure and a left-invariant sub-Riemannian metric. We prove that this metric measure space is locally 4-Ahlfors regular and it is hyperbolic, meaning that it has a non-vanishing 4-capacity at infinity. This implies that the affine-additive group is not quasiconformally equivalent to the Heisenberg group or to the roto-translation group in contrast to the fact that both of these groups are globally contactomorphic to the affine-additive group. Moreover, each quasiregular map, from the Heisenberg group to the affine-additive group must be constant. 
\end{abstract}

\blfootnote{\date{\today}}
\blfootnote{\subjclass{53C17, 30L10}}
\blfootnote{\keywords{Quasiconformal maps, Sub-Riemannian metric, Heisenberg group}}
\blfootnote{\thanks{This research was supported by the Swiss National Science Foundation, Grants Nr. 191978 and 228012}}

\tableofcontents 

\newpage

\section{Introduction and statement of the main result}

Due to work of Heinonen and Koskela, \cite{HK} the theory of quasiconformal mappings has been developed in the setting of general metric measure spaces satisfying some mild regularity properties.
For the related analytic machinery including upper gradients, capacities and Sobolev spaces we refer to the book of  Heinonen, \cite{H}, or the book of Heinonen, Koskela, Shanmungalingam and Tyson \cite{HKST}. 

An important class of examples where these results apply is the geometric setting of sub-Riemannian spaces, including Heisenberg groups. Motivated by Mostow rigidity \cite{MM}, the theory of quasiconformal mappings in the Heisenberg group has been developed by Pansu \cite{Pansu} and Kor\'anyi and Reimann in \cite{KR1} and  \cite{KR2}. This theory is rather advanced, examples of non-trivial quasiconformal maps acting between Heisenberg groups have been constructed as flows of contact vector fields by Kor\'anyi and Reimann \cite{KR1}, \cite{KR2} and by lifting of planar symplectic maps by Capogna and Tang \cite{CT}. Extremal quasiconformal maps that are similar to the planar stretch map,  acting between Heisenberg groups were found by Balogh, F\"assler and Platis \cite{BFP}. Using the flow method of Kor\'anyi and Reimann, Balogh established in \cite{B} the existence of quasiconformal maps between Heisenberg groups distorting the Hausdorff dimension of Cantor sets in a rather arbitrary fashion.  

By a theorem of Darboux, \cite{M} every contact manifold is locally bi-Lipschitz to the Heisenberg group, one would expect that the results of quasiconformal maps could be transposed from the Heisenberg setting to general contact manifolds endowed with a sub-Riemannian metric. However, this turns out not to be the case as contact manifolds are not \textit{globally} quasiconformal to the Heisenberg group. A remarkable example of this has been found by F\"assler, Koskela and Le Donne \cite{FKL} who proved that the sub-Riemannian roto-translation group is not globally quasiconformal to the Heisenberg group, in contrast to the fact, that there exists a global contactomorphism between these spaces. 

In the present paper we consider another natural three dimensional Lie group: the affine-additive group endowed with a sub-Riemannian metric. We prove that it is also globally contactomorphic to both, the Heisenberg group and (by \cite{FKL}) also to the roto-translation group. Howewer, the affine-additive group is not globally quasiconformal to neither the Heisenberg, nor to the roto-translation group. The reason for the non-existence of a global quasiconformal map between these groups is their behaviour at infinity as formulated by Zorich in \cite{Z} (see also \cite{HR}, \cite{FLT}). We prove that the affine additive groups has a non-vanishing 4-capacity at infinity, thus it is hyperbolic in the terminology of \cite{Z}, while both the Heisenberg and the roto-translation groups are parabolic, having a vanishing 4-capacity at infinity. 

To be more precise we define the affine-additive group $(\Aa, \star)$ as the Cartesian product of $\R$ and the hyperbolic right half plane:
$$ \Aa = \R\times \bH_\C^1 \ \text{where} \ 
    \bH_\C^1:=\{(\lambda,t):\lambda>0, t \in \R\} \ 
$$
together with the group law 
$$(a', \lambda', t') \star (a, \lambda, t) = (a' +a, \lambda'\lambda, \lambda't +t')  $$
and the contact $1$-form
\begin{equation*}
\vartheta=\frac{dt}{2\lambda}-da.    
\end{equation*}
For a detailed presentation of the geometric structure of the the affine-additive group $\Aa$ we refer to Section \ref{AA} of this paper. At this point, we can say that the  Carnot-Carath\'eodory distance $d_{\Aa}$ will be defined as the sub-Riemannian distance on $\Aa$ generated by the horizontal vector fields 
\begin{equation*} U =\partial_a+2\lambda\partial_t,  \  V =2\lambda\partial_\lambda, 
\end{equation*}
and a sub-Riemannian metric making $\{U, V\}$ an orthonormal frame. The left-invariant Haar measure on the group $\Aa$ is given by $
d\mu_\Aa=\frac{da\, d\lambda\, dt}{\lambda^2}.
$
The main result of the paper is the following:
\begin{thm} \label{main}
The metric measure space $(\Aa, d_\Aa, \mu_\Aa)$ is a locally $4-$Ahlfors regular space. It is globally contactomorphic to the first Heisenberg group $\H$. The sub-Riemannian manifold, $(\Aa, d_\Aa, \mu_\Aa)$ is $4$ hyperbolic, in particular there is no non-trivial quasiregular map $F: \H \to \Aa$. 
 \end{thm}

The paper is organized as follows: in the Section \ref{Prelim} we fix notation, recall  preliminaries on metric measure spaces and give a sufficient condition on the parabolicity of a metric measure space. In Section \ref{AA} we consider the sub-Riemannian metric of the affine-additive groups in greater detail.  Here we prove that the affine-additive group is globally contactomorphic to the Heisenberg group. In Section \ref{Main} we prove the main result of this paper about the hyperbolicity of the affine-additive group and discuss its consequences. 

\section{Preliminaries on metric measure spaces} \label{Prelim}
We start by recalling some concepts and results on the theory quasiconformal (QC) maps in the setting of general metric measure spaces. For more details we refer to the paper of Heinonen and Koskela \cite{HK}, the book of Heinonen \cite{H} and the book of Heinonen, Koskela, Shanmugalinga and Tyson \cite{HKST}. 

Let us recall that a homeomorphism $f: X\rightarrow Y$ between two metric spaces  $(X,d_X)$ and $(Y,d_Y)$ is called \textit{quasiconformal} if there exists $K\ge 1$ such that 
\begin{equation}\label{metricKqc}
    \limsup_{r\rightarrow0}\frac{\sup_{d_X(p,q)\leq r}d_Y(f(p),f(q))}{\inf_{d_X(p,q)\geq r}d_Y(f(p),f(q))}:=H_f(p)\leq K,
\end{equation}    
for all $p$ in $X$.

A metric measure space is a triple $(X, d_X, \mu_X)$ comprising a non empty set $X$, a distance function $d_X$ and a regular Borel measure $\mu_X$ such that $(X,d_X)$ is a complete, and separable metric space and every metric ball has positive and finite measure. This setting will be our standing assumption throughout this paper. 

Given a point $p \in X$ and a radius $r > 0$, we employ the following
notation for balls:
$$
B_{d_X}(p,r)=\{q\in X:d_X(p,q)<r\}
\text{ and }
\overline{B}_{d_X}(p,r)=\{q\in X:d_X(p,q)\le r\}.
$$
Where it will not cause confusion, we will replace $B_{d_X}(p,r)$ by $B(p,r)$.

A metric measure space $(X, d_X, \mu_X)$ is called \textit{Ahlfors $Q$-regular},
$Q > 1$, if there exists a constant $C\ge1$ such that for all $p\in X$ and $0 < r \leq \diam X$, we have 
\begin{equation}\label{Ahlfors regular}
    C^{-1}r^Q\leq\mu_X(\overline{B}_{d_X}(p,r))\leq Cr^Q.
\end{equation} 

Further, we say that $(X, d_X, \mu_X)$ is \textit{locally Ahlfors $Q$-regular}, if for every compact
subset $V \subset X$, there is a constant $C \geq 1$ and a radius $r_0>0$ such that for
each point $p \in V$ and each radius $0 < r \leq r_0$ we have 
\begin{equation}\label{local Ahlfors regular}
    C^{-1}r^Q\leq\mu_X(\overline{B}_{d_X}(p,r))\leq Cr^Q.
\end{equation}

An important geometric quantity in the theory of quasiconformal mappings is the  $Q$-modulus of a curve family. Let us recall, that if  $\Gamma$ be a family of curves in the metric measure space $(X,d_X,\mu_X)$, a  Borel function $\rho: X \rightarrow [0, \infty]$
is said to be \textit{admissible} for $\Gamma$ if for every rectifiable $\gamma \in \Gamma$, we have 
\begin{equation*}
1\leq \int_\gamma \rho\, d\ell_X \,.
\end{equation*}
Such a $\rho$ shall be also called a density and the set of all densities shall be denoted by ${\rm Adm}(\Gamma)$.
If $Q>1$ then the $Q$-\textit{modulus} of $\Gamma$ is
\begin{equation*}
    \Mod_Q(\Gamma)=\inf_{\rho\in{\rm Adm}(\Gamma)}\int_X \rho^Q \,d\mu_X.
\end{equation*}
It follows immediately from this definition that if  $\Gamma_0$ and $\Gamma$ are two curve families such that each curve $\gamma \in \Gamma$ has a sub-curve
$\gamma_0 \in \Gamma_0$, then 

\begin{equation} \label{subcurve-mod}
\Mod_Q(\Gamma)\leq\Mod_Q(\Gamma_0).
\end{equation} 

Let us recall that by Theorem 3.8 in \cite{KW}, if  $(X,d_X,\mu_X)$ and $(Y,d_Y,\mu_Y)$ are separable, locally finite metric measure
spaces that are both locally Ahlfors $Q$-regular for some given $Q>1$ and  $f : X \rightarrow Y$ is a quasiconformal map then there exists $H\geq 1$ such that 
\begin{equation} \label{mod-qc}
    \Mod_Q (\Gamma)\leq H \,\Mod_Q (f(\Gamma)), 
\end{equation}
for every curve family $\Gamma$ in $X$, i.e., the $\Mod_Q$ is quasi-preserved by quasiconformal maps. 

For two disjoint compact sets $E, F \subset X$  we consider the number $\Mod_Q(E, F) =  \Mod_Q (\Gamma)$ 
where $\Gamma$ is the set of all rectifiable curves connecting $E$ and $F$. If $x_0 \in X$ is a fixed point and $0< r < R < \diam X$, $E= \partial B(x_0, r) $ and $F = \partial B(x_0, R)$ then the quantity $\Mod_Q(E, F)= \Mod_Q(\mathcal{D}(r, R))$ is the so called modulus of the ring domain 
$$\mathcal{D}(r, R) = \{ x \in X: r<d(x, x_0)< R \} .$$ 

The following definition is a reformulation in the setting of metric spaces of the corresponding concept by  Zorich \cite{Z}. For related results we refer also to Holopainen and Rickman \cite{HR}, Coulhon, Holopainen and Saloff-Coste \cite{CHSC}, F\"assler, Lukyanenko and Tyson \cite{FLT}.

\begin{defn} The metric measure spaces $(X, d_X, \mu_X)$ is $Q$-parabolic if and only if for some $x_0 \in X$ and $R_0>0$ we have 
\begin{equation} \label{conf-par}
  \lim_{R \to \infty } \Mod_Q(\mathcal{D}(R_0, R))  =0. 
\end{equation}
 Otherwise we call $(X, d_X, \mu_X)$ $Q$-hyperbolic.  
\end{defn}

Let us note, that parabolicity of a metric measure space is a property about the behaviour of the space at infinity. In particular, $(X, d_X, \mu_X)$ is $Q$-parabolic if and only if for \textit{any} $x_0 \in X$ and $R_0>0$ \eqref{conf-par} holds. We remark that $Q$-parabolicity of a metric measure space can be defined equivalently by capacity of condensers (see Section 7 in \cite{Z} and Definition 4.5.4 in \cite{FLT}).

The following sufficient condition seems to be known to experts, however we could not locate a precise reference and  we include it for the sake of completeness.

\begin{prop} \label{Q-par-prop} 
Let $Q>1$ and $(X, d_X, \mu_X)$ be a metric measure space such that there exists $x_0 \in X$,  $R_0>0$ and $K>0$ such that for all $R>R_0$ we have 
\begin{equation} \label{measure-growth}
    \mu_X(B(x_0, R)) \leq K R^Q.
\end{equation}
 Then $(X, d_X, \mu_X)$ is $Q'$-parabolic for any $Q'\geq Q$.     
\end{prop}
\begin{proof}
We shall  consider the ring domain $\mathcal{D}(R_0, R) = \{ x \in X: R_0 < d_X(x, x_0)< R\}$ for $R>R_0$. Our purpose is to show that 
\begin{equation*}
 \lim_{R \to \infty} \Mod_{Q'}(\mathcal{D}(R_0, R)) =0 .  
\end{equation*}
To do this, we consider the integer $N \in \N$ defined by the property that $2^N R_0 \geq R >2^{N-1}R_0$. Note, that 
if $R \to \infty$, then $N\to \infty$. Consider the density 
\begin{equation*}
    \rho_N(x)=\left\{\begin{matrix}
        \frac{3}{N}\cdot \frac{1}{d_X(x_0, x)} & &\text{ if }\; x \in \mathcal{D}(R_0, R)\\
        \\
        0& &\text{otherwise}.
    \end{matrix}\right.
\end{equation*}
Let us check that the $\rho_N$ is an admissible density for the curve family $\Gamma$ connecting $\partial B(x_0, R_0)$ and $\partial B(x_0, R)$. To do so we consider the integers $1< k < N$ and denote by $B_k= B(x_0, 2^k R_0)$ and $D_k= B_k \setminus B_{k-1}$. For $\gamma \in \Gamma$ denote by 
$\gamma_k = D_k \cap \gamma$. By this notation, we observe that the length  of $\gamma_k$, $\ell_X(\gamma_k) \geq 2^{k-1} R_0$ and if $x\in \gamma_k$, then $\rho_N(x) \geq \frac{3}{N}\cdot \frac{1}{2^kR_0}$. 
Using this information we can write
$$ \int_{\gamma} \rho_N\,d\ell_X \geq \sum_{k=2}^{N-1} \int_{\gamma_k} \rho_N \,d\ell_X\geq \sum_{k=2}^{N-1} \frac{3}{N}\cdot \frac{1}{2^kR_0} \ell(\gamma_k) \geq \frac{3(N-2)}{2N} \geq1, $$
if $N\geq6$.
Note, that by our assumption on the upper of the measure \eqref{measure-growth} we have that $\mu_X(B_k) \leq K 2^{kQ}R_0^Q$. Using this upper estimate on the measure of $B_k$, the assumption $Q'\geq Q$ and the fact that for $x\in B_k$ we have $\rho(x) \leq \frac{3}{N}\frac{1}{2^{k-1} R_0}$, we can estimate 
\begin{align*}
& \Mod_{Q'}\mathcal{D}(R_0, R)  \leq \int_{\mathcal{D}(R_0, R)} \rho_N^{Q'} d\mu_X \leq \sum_{k=1}^N \int_{D_k} \rho_N^{Q'} d\mu_X \leq \sum_{k=1}^N \int_{B_k} \left(\frac{3}{N}\frac{1}{2^{k-1} R_0}\right)^{Q'} d\mu_X= \\ & = \sum_{k=1}^N \left(\frac{3}{N}\frac{1}{2^{k-1} R_0}\right)^{Q'} \mu_X(B_k) \leq K\left(\frac{6}{N}\right)^{Q'} R_0^{Q-Q'} \sum_{k=1}^N2^{k(Q-Q')} \to 0  \ \ \text{as} \ \ N \to \infty,
\end{align*}
 Since $R\to \infty $  implies that $N \to \infty $ we obtain the statement. 
\end{proof}

As expected, our next statement is a formulation of the fact that a parabolic metric measure space cannot be quasiconformally equivalent to an hyperbolic one. In order to formulate the statement we recall that a metric space is proper, if its closed metric balls are compact. 

\begin{thm} Let $Q>1$ and let $(X, d_X, \mu_X)$, $(X', d_{X'}, \mu_{X'})$ be two locally Ahlfors, $Q$-regular metric measure spaces. Assume that both spaces are proper and  $(X, d_X, \mu_X)$ 
 is hyperbolic and  $(X', d_{X'}, \mu_{X'})$ is a parabolic space.  Then there is no QC map $f: X \to X'$.  
\end{thm}

\begin{proof}
Assume by contradiction that there is a QC map $f:X \to X'$. Since $(X, d_X, \mu_X)$ is assumed to be hyperbolic, there exist a point $x_0 \in X$,
$R_0 >0$, a sequence $R_n \to \infty$, and a number $M>0$ such that 
$$ \Mod_Q(\Gamma_n) \geq M>0,\, n\geq n_0,$$
where $\Gamma_n$ is the set of curves connecting $\partial B_X(x_0, R_0)$ and 
$\partial B_X(x_0, R_n)$. By the relation \eqref{mod-qc} there exists $H\geq 1$ such that 
$$ \Mod_Q (f(\Gamma_n)) \geq \frac{\Mod_Q(\Gamma_n)}{H} \geq \frac{M}{H}>0.$$
Let us denote by $y_0 = f(x_0) \in X'$. Since $X$ is proper, $\bar{B}_X(x_0, R_0)$ is compact and thus $f(B_X(x_0, R_0))$ is bounded in $X'$. We conclude that there exists a number $R_0'>0$ such that $ f(B_X(x_0, R_0)) \subseteq B_{X'}(y_0, R_0')$. Let us denote by 
$$ R_n' := \min \{ d_{X'}(f(x_0), f(x)): x \in \partial B_X(x_0, R_n) \}.$$

We claim that $R_n' \to \infty$. For otherwise, we find a sequence $x_n\in X$ with $d_X(x_0, x_n) = R_n$ such that $d_X'(f(x_0), f(x_n)) \leq M'$ for some fixed constant $M'>0$. Since the space $X'$ is a proper metric space, we obtain that (up to a subsequence) $f(x_n) \to y$ for some $y\in X'$. 
Let us denote by $x_1 = f^{-1}(y) \in X$ the preimage of $y$. Since $f$ is a homeomorphism we have that $f(B_X(x_1, r))$ is a neighborhood of $y\in X'$ for 
any fixed $r>0$. Since $f(x_n) \to y$ we must have that for $n$ large enough $f(x_n) \in f(B_X(x_1, r))$, which is a contradiction to the injectivity of $f$.

Let us note that any curve in $f(\Gamma_n)$ has a sub-curve connecting $\partial B_{X'}(y_0, R_0')$ and $\partial B_{X'}(y_0, R_n')$. This implies by \eqref{subcurve-mod} that 
$$ \Mod_Q(D(R_0', R_n'))\geq \Mod_Q(f(\Gamma_n)) \geq \frac{M}{H},$$
which is a contradiction to the parabolicity of $(X', d_{X'}, \mu_{X'})$, concluding the proof.  
\end{proof}

The metric spaces considered in this paper are 3-dimensional Lie groups $\bG$ with group multiplication $\star$. We shall assume that $\bG$ is equipped with a left-invariant contact form $\vartheta_\bG$. Using this contact form we define a left-invariant sub-Riemannian metric on $\bG$ as follows. 

The kernel $\ker\vartheta_\bG =\mathcal{H}_\bG$ is a two dimensional subbundle of the tangent bundle ${\rm T}_\bG$. If $X$ and $Y$ are left-invariant vector fields such that $\mathcal{H}_\bG={\rm span}\{X,Y\}$ then 
 a left-invariant sub-Riemannian metric $\langle\cdot,\cdot\rangle_\bG$ is considered in $\mathcal{H}_\bG$, making  $\{X,Y\}$ an orthonormal basis of $\mathcal{H}_\bG$. 

A curve $\gamma:[a,b]\to\bG$ , $\gamma=\gamma(s)$ shall be called horizontal if $\dot{\gamma}(s)\in\ker(\vartheta_{\bG})_{\gamma(s)}$ for almost every $s\in[a,b]$. Then, the horizontal velocity of $\gamma$ is 
$$|\dot\gamma(s)|_\bG=\sqrt{\langle\dot\gamma(s),X_{\gamma(s)}\rangle_{\bG}^2+\langle\dot\gamma(s),Y_{\gamma(s)}\rangle_{\bG}^2}.$$
The horizontal length of $\gamma$ is $$\ell_\bG(\gamma)=\int_a^b|\dot\gamma(s)|_\bG\,ds.$$
The corresponding sub-Riemannian or Carnot-Carath\'eodory distance $d_{\bG}$ associated to the sub-Riemannian metric $\langle\cdot,\cdot\rangle_\bG$ is defined in $\bG$ as follows: let $p, q  \in\bG$ and cosider the family $\Gamma_\bG(p,q)$ of horizontal curves $\gamma:[a,b]\to\bG$ such that $\gamma(a)=p$ and $\gamma(b)=q$. Then
\begin{equation}\label{CC distance}
    d_{\bG}(p,q)=\inf_{\gamma\in \Gamma_\bG(p,q)}\{\ell_\bG(\gamma) \},
\end{equation}
We remark that the above definition only depends on the values of $\langle\cdot,\cdot\rangle_\bG$ on $\mathcal{H}_\bG$.
Moreover, since $\mathcal{H}_\bG$ is completely non integrable, the distance $d_\bG$ is finite, geodesic, and
induces the manifold topology (see e.g. \cite{M}). This will make the space $(X, d_X) = (\bG, d_{\bG})$ a metric space. We consider the measure $\mu_X= \mu_{\bG}$ induced by the contact form $\vartheta_\bG$ by $ \mu_{\bG}= \vartheta_\bG\wedge d\vartheta_\bG$ (up to a multiplicative constant different from $0$) that is also left-invariant and gives our metric measure space $(\bG, d_{\bG}, \mu_{\bG})$. 

A well-known example of such a structure is  the \textit{first Heisenberg group}  $\H$. Its underlying manifold is $\C\times\R$ with coordinates $(z=x+iy,t)$ and the group multiplication $\star$ is given by 
$$p'\star p=\left(z'+z,t'+t+2\Im(\overline{z}'z)\right)$$
for every $p=(z,t)$ and $p'=(z',t')$ in $\C\times\R$.  

The \textit{contact form} of $\H$ is given by:
$$\vartheta_\H=dt+2\Im(\overline{z}dz)=dt+2(xdy-ydx).$$
The horizontal bundle $\mathcal{H}_\H$ of the tangent bundle is spanned by the vector fields
$$X=\partial_x+2y\partial_t,\;Y=\partial_y-2x\partial_t\,.$$
Denote the sub-Riemannian metric in $\H$ by $\langle\cdot,\cdot\rangle_\H$ making $\{X,Y\}$ an orthonormal frame. The horizontal length of a curve $\gamma=\gamma(s)$, $s\in[a,b]$, $\gamma(s)=(z(s),t(s))$ is 
$$
\ell_\H(\gamma)=\int_a^b |\dot z(s)| \, ds.
$$ 
Denote also the corresponding
Carnot-Carath\'eodory distance by $d_\H$.
The measure $\mu_\H$ is the bi-invariant Haar measure for $\H$ and it coincides with  the 3-dimensional Lebesgue measure in $\C\times\R$ denoted with $\mathcal{L}^3$. It turns out that the $(\H, d_{\H}, \mu_{\H})$ is a parabolic, $4$-Ahlfors regular metric measure space. It follows from Proposition \ref{Q-par-prop} that the metric measure space 
$(\H, d_{\H}, \mu_{\H})$ is 4-parabolic. We note, that there is an elaborate theory of QC maps on the Heisenberg group (see e.g. \cite{Pansu}, \cite{KR1}, \cite{KR2}, \cite{CT}, \cite{BFP}, \cite{P}). It is therefore of interest to identify those sub-Riemannian Lie groups that are QC equivalent to the Heisenberg group.

The second example is the \textit{roto-translation group}  $\RT$ (see Chapter 3 in \cite{CDPT} and \cite{FKL}). Its underlying manifold is $\C\times\R$ with coordinates $p=(z=x+iy,t)$ and the group multiplication $\ast$ is given by 
$$p'\star p=\left(e^{it'}z+z',t'+t\right)\in\C\times\R$$
for every $p=(z,t)$ and $p'=(z',t')$ in $\C\times\R$.

The \textit{contact form} of $\RT$ is given by: 
$$\vartheta_\RT=\sin t\,dx-\cos t\,dy.$$
The horizontal bundle $\mathcal{H}_\RT$ of the tangent bundle is spanned by the vector fields

$$X=\cos t\, \partial_x+ \sin t\, \partial_y,\;Y=\partial_t\,.$$

Denote the sub-Riemannian metric in $\RT$ by $\langle\cdot,\cdot\rangle_\RT$ making $\{X,Y\}$ an orthonormal frame. The horizontal length of a curve $\gamma=\gamma(s)$, $s\in[a,b]$, $\gamma(s)=(z(s),t(s))$ is 
$$
\ell_\RT(\gamma)=\int_a^b |\dot \gamma(s)|_\RT \, ds,
$$ 
where
$$
|\dot \gamma(s)|_\RT=\sqrt{(\dot x(s)\cos t(s)+\dot y(s)\sin t(s))^2+\dot t(s)^2}.
$$
Denote also the corresponding
Carnot-Carath\'eodory distance by $d_\RT$.
The measure $\mu_{\RT}$ is again the bi-invariant Haar measure of $\RT$ and it is the 3-dimensional Lebesgue measure in $\C\times\R$.

Later on we will need the following (Lemma 5.5 in \cite{FKL}):
\begin{prop}\label{RT contact equiv to H}
   The manifolds $(\RT , \vartheta_\RT)$ and $(\H, \vartheta_\H)$ are globally contactomorphic: i.e.,
there is a diffeomorphism $f:\RT\to\H$ such that $$f^*\vartheta_\H=\sigma \,\vartheta_\RT,$$ where $\sigma:\RT\to\R$ is a nowhere vanishing smooth function.
\end{prop}

We will also need (Corollary 5.9 in \cite{FKL}): 
\begin{prop}\label{large scale RT}
There exists $R_0 > 0$, and $C_0 > 0$ such that if $B_{\mathcal{RT}}(e_{\mathcal{RT}},r)$ is the open $CC$-ball of centre $e_{\mathcal{RT}}$ and radius $r$ then:
\begin{equation}\label{RT - N=3}
    \mathcal{L}^3(B_{\mathcal{RT}}(e_{\mathcal{RT}},r))\leq C_0 r^3, \text{ for all } r\geq R_0.
\end{equation}
\end{prop}
The remarkable result of F\"assler, Koskela and Le Donne states that in contrast to the fact that both spaces $(\H, d_{\H}, \mu_{\H}) $ and $(\RT, d_{\RT}, \mu_{\RT})$ are 4-parabolic and by \ref{RT contact equiv to H} locally bi-Lipschitz equivalent but they are still not QC equivalent (see Corollary 1.2 in \cite{FKL}).

\section{The affine-additive group as a metric measure space} \label{AA}
The main subject of this paper is the affine-additive group, which we describe below. In particular, after introducing the group, we discuss its sub-Riemannian structure. For more details about the affine-additive group, we refer to \cite{Bubani}.

Our starting point is the hyperbolic plane, defined as 
\begin{equation*}
    \bH_\C^1:=\{\zeta=\xi+i\eta\in\C:\xi>0\} \ \text{with the Riemannian metric }  
g=\frac{|d\zeta|^2}{4\xi^2}=\frac{d\xi^2+d\eta^2}{4\xi^2}.
\end{equation*}

We consider affine transformations on $\bH_\C^1$, composed by  dilations $D_\lambda$, $\lambda>0$, defined by
$
D_\lambda(\zeta)=\lambda\zeta,
$
and translations $T_t$, $t\in\R$, defined by $T_t(\zeta)=\zeta+it$, for $\zeta\in \bH_\C^1$ resulting in maps of the form
$$
M(\lambda,t)(\zeta)=(T_t\circ D_\lambda)(\zeta)=\lambda\zeta+it.
$$
It is clear that $\bH_\C^1$ is in bijection with the set of transformations of the above form: to each point $\xi+i\eta$ we uniquely assign the transformation $M(\xi,\eta)$. Therefore we define a group structure on $\bH_\C^1$ by considering the composition of any two transformations $M(\lambda',t')$ and $M(\lambda,t)$:
\begin{eqnarray*}
(M(\lambda',t')\circ M(\lambda,t))(\zeta) = M(\lambda',t')(\lambda\zeta+i t) =\lambda'\lambda\zeta+i(\lambda't+t')= M(\lambda'\lambda,\;\lambda't+t')(\zeta).
\end{eqnarray*}
To sum up, (compare to Section 4.4.2 in \cite{Petersen}) the group operation on $\bH_\C^1$ is given by
\begin{equation}\label{hyp-group}
    (\lambda',t')\cdot(\lambda,t)=(\lambda'\lambda,\;\lambda't+t').
\end{equation}

We wish to extend the previous construction over the space 
$
\R\times\bH_\C^1
$. 
We define the group operation as follows: if $p'=(a',\lambda',t')$ and  $p=(a,\lambda,t)$ are points of $\R\times\bH_\C^1$, then
\begin{equation}\label{eq-mult}
p'\star p=(a'+a,\lambda'\lambda,\lambda't+t'),
\end{equation}
which is again a point in $\R\times\bH_\C^1$. This group operation is the group operation of the Cartesian product of the additive group $(\R,+)$ and the group $(\bH_\C^1,\cdot)$, where $\cdot$  is as in \eqref{hyp-group}.
\begin{defn}
The pair $\Aa=(\R\times\bH_\C^1,\;\star)$ shall be called the affine-additive group.
\end{defn}

We define a $1$-form on $\Aa$ as follows: 
\begin{equation}\label{contact form omega}
\vartheta=\frac{dt}{2\lambda}-da.    
\end{equation}
Since $d\vartheta=\frac{1}{2\lambda^2}dt\wedge d\lambda$ we obtain 
$\vartheta\wedge d\vartheta=\frac{dt\wedge da\wedge d\lambda}{2\lambda^2}$
and thus $(\Aa,\vartheta)$  is a contact manifold.  In what follows we identify the left invariant vector fields and define a left invariant sub-Riemannian metric on the group $\Aa$. 

\begin{prop}
The vector fields
\begin{equation}\label{LIVF} U =\partial_a+2\lambda\partial_t,\nonumber  \  V =2\lambda\partial_\lambda, \ W =-\partial_a\nonumber
\end{equation}
are left-invariant and form a basis for the  
tangent bundle $T(\Aa)$ of $\Aa$. They satisfy the following Lie bracket relations:
\begin{equation}\label{livf_rels}
    [U,W]=[V,W]=0 \;\text{ and }\; [U,V]=-2(U+W);
\end{equation} 
Moreover, the left-invariant measure for $\Aa$ is
$
d\mu_\Aa=\frac{da\, d\lambda\, dt}{\lambda^2}.
$
\end{prop}
\begin{proof}
By the definition of $U, V$ and $W$ we have the relations:
\begin{equation}\label{eq-basis}
\partial_a=-W,\quad\partial_\lambda=\frac{U}{2\lambda},\quad \partial_t=\frac{U+W}{2\lambda},
\end{equation}
and thus $\{U,V,W\}$ is a basis for $T(\Aa)$.
To verify that $U,V$ and $W$ are left-invariant, it suffices to consider their values  at $e=e_\Aa=(0,1, 0)$. For this, let 
\begin{equation*}
    \{U_e=(\partial_a+2\partial_t)_{|e},\;V_e=(2\partial_\lambda)_{|e},\; W_e=(-\partial_a)_{|e}\}
\end{equation*}
be the basis for $T_e(\Aa)$. 
If we fix a point $p'=(a',\lambda',t')\in\Aa$ we can consider the left translation on $\Aa$ given by
\begin{equation}\label{eq-lefttr}
    L_{p'}(p)=p'\star p=(a'+a,\lambda'\lambda,\lambda't+t'),
\end{equation} 
the  Jacobian matrix of the derivative $(L_{p'})_{*,p}$ of $L_{p'}$ evaluated at $p$ is
\begin{equation*}
(DL_{p'})_p=\left[\begin{array}{ccc}
     1&0 &0 \\
     0&\lambda' &0 \\
     0&0 & \lambda'
\end{array}\right].
\end{equation*}

Then by $(L_p)_{*,e}:T_e(\Aa)\rightarrow T_p(\Aa)$ we have 
\begin{equation}
    (L_{p})_{*,e}\,(\partial_a+2\partial_t)_{|e}=U_p,\
    (L_{p})_{*,e}\,(2\partial_\lambda)_{|e}=V_p,\
    (L_{p})_{*,e}\,(-\partial_a)_{|e}=W_p,
\end{equation}
proving the first claim. 
 The verification of the Lie bracket relations and the left-invariance of $\mu_{\Aa}$ are left to the reader. 
\end{proof}
 Note, that $\vartheta(U)=\vartheta(V)=0$ and thus, the horizontal bundle of $\Aa$ is $\mathcal{H}_{\Aa}=\text{Span}\{U,V\}$.
The sub-Riemannian structure in $\Aa$ is defined by a sub-Riemannian metric on $\mathcal{H}_{\Aa}$ making $\{ U, V\})$ an orthonormal basis. 
In order to define the sub-Riemannian or Carnot-Carath\'eodory distance on $\Aa$ let $\gamma:[0,1]\rightarrow\Aa$, $\gamma(s)=(a(s),\lambda(s),t(s))$ be a (piecewise) $C^1$ curve. Its tangent vector at $\gamma(s)$ is 
\begin{equation*}
    \dot\gamma(s)=\frac{\dot t(s)}{2\lambda(s)}U_{\gamma(s)}+\frac{\dot \lambda(s)}{2\lambda(s)}V_{\gamma(s)}+\left(\frac{\dot t(s)}{2\lambda(s)}-\dot a(s)\right)W_{\gamma(s)}.
\end{equation*}
The curve $\gamma$ is a horizontal curve if and only if $\dot\gamma(s)\in\ker(\vartheta)_{\gamma(s)}$ for almost every $s\in[0,1]$. This is equivalent to the ODE
\begin{equation}
    \frac{\dot t(s)}{2\lambda(s)}-\dot a(s)=0, \text{ a.e. } s\in[0,1].
\end{equation}
It follows that for a horizontal curve
$$
\dot\gamma(s)=\frac{\dot t(s)}{2\lambda(s)}U_{\gamma(s)}+\frac{\dot \lambda(s)}{2\lambda(s)}V_{\gamma(s)}\in({\mathcal H}_{\Aa})_{\gamma(s)}.
$$
The horizontal velocity $|\dot\gamma|_H$ of $\gamma$ is now defined by the relation
\begin{equation}\label{horiz speed}
|\dot\gamma|_H=\left(\langle\dot\gamma,U\rangle_{\Aa}^2+\langle\dot\gamma,V\rangle_{\Aa}^2\right)^{1/2}=\frac{\sqrt{\dot \lambda^2+\dot t^2}}{2\lambda}.
\end{equation}
Here, $\langle\cdot,\cdot\rangle_{\Aa}$ is the sub-Riemannian metric on $\mathcal{H}_{\Aa}$. Let $\pi:\Aa\to\bH^1_\C$ denote the canonical projection given by $\pi(a,\lambda,t)=(\lambda,t),\,(a,\lambda,t)\in\Aa$, the horizontal length of $\gamma$ is then given by
\begin{equation}
 \ell(\gamma)=\int_0^1 \frac{\sqrt{\dot \lambda^2+\dot t^2}}{2\lambda}ds=\ell_h(\pi\circ \gamma),
\end{equation}
where $\ell_h(\pi\circ \gamma)$ is the hyperbolic length of the projected curve $\pi\circ\gamma$ in $\bH^1_\C$. It is straightforward to prove that the horizontal length is invariant under left-translations. 

Conversely, if $\tilde\gamma$ is a $C^1$ curve in $\bH^1_\C$, $\tilde\gamma(s)=(\xi(s),\eta(s))$, $s\in[0,1]$, passing from a point $q_0=\gamma(s_0)$, then the curve $\gamma:[0,1]\to\Aa$ given by $\gamma(s)=(a(s),\lambda(s),t(s))$, where
\begin{equation*}
a(s)=\int_{s_0}^s\frac{t(u)}{2\lambda(u)}\;du+a_0,\quad \lambda(s)=\xi(s),\quad t(s)=\eta(s),
\end{equation*}
is a horizontal curve passing from a point $p_0=(a_0,q)$ in the fibre of $q$.

The corresponding Carnot-Carath\'eodory distance $d_{\Aa}$ associated to the sub-Riemannian metric $\langle\cdot,\cdot\rangle$ is defined for all $p, q  \in\Aa$ as follows:
\begin{equation}\label{CC-distance}
    d_{\Aa}(p,q)=\inf_{\gamma\in \Gamma_\Aa}\{\ell(\gamma) \},
\end{equation}
where $\Gamma$ is the following family of horizontal curves: 
$$
\Gamma_\Aa=\{\gamma\in C^1([0,1],\Aa):\; 
       \gamma(0)=p,\;
       \gamma(1)=q\}.
$$
We recall that the distance $d_{\Aa}$ is finite, geodesic and
induces the manifold topology. Our main object of study is the metric measure space $(\Aa, d_{\Aa}, \mu_{\Aa})$. 
It is well known, that by Darboux theorem each three dimensional contact manifold is locally contactomorphic to the Heisenberg group. Our next statement is a stronger, global version of this fact. 

\begin{prop}\label{global-cont AA Heis}
    The manifolds $(\Aa , \vartheta)$ and $(\H, \vartheta_\H)$ are globally contactomorphic. That is: there exists a smooth bijective map $g:\H \to \Aa$ such that $g^*\vartheta= \nu \vartheta_{\H}$ for some non-vanishing smooth function $\nu: \H \to \R$. In particular the metric spaces $(\H, d_{\H})$ and $(\Aa, d_{\Aa})$ are locally bi-Lipschitz equivalent. 
\end{prop}

\begin{proof}
We define the smooth contactomorphism $g:(\H,\vartheta_\H)\rightarrow(\Aa,\vartheta)$ explicitely by the formula
\begin{equation}\label{Heis contact to Aa}
    g(x,y,t)=
         \left(x e^{-y},\;
          e^{y},\;
         \frac{1}{2}(t-2xy+4x)
    \right) \text{ for } (x,y,t)\in\H.
\end{equation}
Clearly, $g$ is a smooth diffeomorphism between $\H$ and $\Aa$. Its inverse map $g^{-1}: \Aa \to \H$ is given by
\begin{equation*}
   g^{-1}(a,\lambda,t)=\left(a\lambda,\; 
          \ln \lambda,\;
          2t+2a\lambda(\ln \lambda -2)
    \right) \text{ for } \ (a,\lambda,t)\in\Aa. 
\end{equation*}
To check the contact property of $g$ we compute directly:
\begin{eqnarray*}
 g^*\vartheta&=&\frac{(1/2)dt-xdy-ydx+2dx}{2e^y}-e^{-y}dx+xe^{-y}dy= \frac{dt+2xdy-2ydx}{4e^y}= \frac{1}{4e^y}\vartheta_\H.
\end{eqnarray*}
\end{proof}

Combining Proposition \ref{RT contact equiv to H} and Proposition \ref{global-cont AA Heis} we deduce
\begin{prop}\label{RT AA H contact equiv}
    The manifolds $(\RT,\vartheta_\RT)$, $(\H,\vartheta_\H)$ and $(\Aa,\vartheta)$ are all globally contactomorphic to each other.
\end{prop}

Another consequence of Proposition \ref{global-cont AA Heis} is the following: 
\begin{prop}\label{RT and AA loc Ahl 4-reg}
The metric measure space $(\Aa, d_{\Aa}, \mu_{\Aa})$ is locally Ahlfors $4$-regular.
\end{prop}

\begin{proof}
We are going to prove a stronger property for $(\Aa,d_\Aa,\mu_\Aa)$; actually, that there exist a $C\geq1$ and an $r_0>0$ such that 
\begin{equation}\label{strong loc 4 Alhfors reg}
    C^{-1}r^4\leq\mu_ \Aa(\overline{B}_{d_ \Aa}(p,r))\leq C r^4,
\end{equation}
for all $0<r\leq r_0$ and for all $p\in\Aa$.
Due to the left-invariance of both the sub-Riemannian distance $d_\Aa$ and  the  measure $\mu_\Aa$, it suffices to prove \eqref{strong loc 4 Alhfors reg} for balls $B_\Aa(e,r)$ centered at the neutral element $e=e_\Aa$.
We have that $(\H,\vartheta_\H)$ and $(\Aa,\vartheta)$ are globally contactomorphic thanks to Proposition \ref{global-cont AA Heis}, so let us consider the map 
$$
  g:(\H,d_\H,\L^3)\rightarrow(\Aa,d_\Aa,\mu_\Aa), \quad  g^*\vartheta=\nu \vartheta_\H,
$$
 given in \eqref{Heis contact to Aa}. Since $g(e_\H) = e$, where $e_\H=(0,0,0)$ is the neutral element of $\H$ and $g:\H\rightarrow \Aa$ is locally bi-Lipschitz, we have the inclusions 
$$ g(B_{\H} (e_\H, L^{-1}r)) \subseteq B_{\Aa}(e, r) \subseteq g(B_{\H} (e_\H, Lr))$$
for some fixed number $L \geq 1$ and any $0 \leq r\leq 1$.
Since $g^*\mu_{\Aa}= \nu^2 \mu_{\H}= \nu^2 \L^3$ (up to multiplicative constants different from $0$) and $\L^3(B_{\H}(e_\H,r)) = C r^4$ for some fixed constant $C>0$, the claim follows. 
\end{proof}

Due to Proposition \ref{RT AA H contact equiv} we obtain the same result for the roto-translation group $\RT$ as well, and moreover, according to Proposition \ref{Q-par-prop}, it follows that 

\begin{prop}\label{H,RT loc 4-reg, 4-parab}
The metric measure spaces $(\H,d_\H,\mu_\H)$ and 
$(\RT, d_{\RT}, \mu_{\RT})$ are locally $4$-Ahlfors regular and $4$-parabolic.     
\end{prop}

Another consequence of Proposition \ref{global-cont AA Heis} is the following result about the existence of non-smooth QC maps of $\Aa$ that distort the Hausdorff dimension $\dim_H$ of certain Cantor sets in $\Aa$ in an arbitrary fashion:

\begin{prop} \label{h-dim}For any $s,t$, $0<s<t<4$ there exist Cantor sets $C_s \subset \Aa$ and  $C_t \subset \Aa$ such that $\dim_H(C_s) = s$ and $\dim_H(C_t) = t$ and a QC map $F: \Aa \to  \Aa$ such that $F(C_s) = C_t $. 
\end{prop}
\begin{proof}
 The proof is based on the corresponding result in \cite{B} for the case of the Heisenberg group. In fact Theorem 1.1 in \cite{B} states that if $0<s<t<4$ there exist Cantor sets $K_s \subset \H$ and $K_t \subset \H$ and a QC map $G: \H \to \H$ such that  $\dim_H(K_s) = s$, $\dim_H K_t =t$, $K_s \subset B_{\H}(e, 1)$, $K_t \subset B_{\H}(e, 1)$, $G(K_s) = K_t$ and $G = id_{\H}$ outside of 
 $B_{\H}(e,1)$. 
 It is easy to see that the map $F: \Aa \to \Aa$ defined by 
 $F= g\circ G \circ g^{-1}$ satisfies the properties in the statement for the sets $C_s= g(K_s)$ and $C_t = g(K_t)$.
\end{proof}

More results on quasiconformal maps defined on the affine-additive group $\Aa$, including methods of constructions such maps and study of their extremality  will be contained in the forthcoming paper \cite{BBP} of the authors as well as in the  dissertation of the second author \cite{Bubani}.

\section{Proof of the main result} \label{Main}

Let us observe, first, that hyperbolicity of a metric measure space $(X, d_X, \mu_X)$ holds, if there exists a compact set $E\subset X$ and sequence of compact sets $F_n \subseteq X$ such that 
$$\dist(E, F_n):= \inf \{ d(x,y): x \in E, y\in F_n \} \to \infty$$ and
$$ \liminf_{n \to \infty} \Mod_Q(Q, F_n) > 0. $$
To do this, let us pick $x_0 \in E$. 

We shall consider $R_n = \inf \{ d(x_0,y): y\in F_n \}$. Note, that  $R_n \to \infty $ and any curve connecting $E$ and $F_n$ must have a sub-curve connecting $\partial B(x_0, R_0)$ and $\partial B(x_0, R_n)$. Thus, by \eqref{subcurve-mod} we have the inequality 
$$ \Mod_Q(\mathcal D(R_0, R_n)) \geq \Mod_Q(E, F_n).$$
Since $ \liminf_{n \to \infty} \Mod_Q(E, F_n) > 0$ we obtain that $(X, d_X, \mu_X)$ is hyperbolic.

The idea of the proof of Theorem \ref{main} is to construct compact sets $E$ and $F_n$ in $\Aa$ with the above properties. This is explicitly done as follows:

Let $n\in\N$, $n\ge 2$. We define
\begin{align*}
    E&=\{(a,1,t)\in\Aa:a\in[-1,1]\text{ and } t\in[-1,1]\},\\
    F_n&=\{(a,\frac{1}{n},t)\in\Aa:a\in[-1,1]\text{ and } t\in[-1,1]\}.
\end{align*}
Next, for each such $n$ we define the following curve families of piecewise smooth horizontal curves:
\begin{align}\label{Gamma_n}
    \Gamma_n=\{\gamma,&\gamma:[0,1]\rightarrow\Aa \;\text{ such that} \;
         \gamma(0)\in E  \;\text{and}\;
         \gamma(1)\in F_n\}.
\end{align}
The following estimate holds.
\begin{prop}\label{Curve mod lemma}
    There exists some $M>0$ such that $\Mod_4(\Gamma_n) >M$ for all $n\in\N$, $n\geq2$.
\end{prop}
\begin{proof}
We consider the sub-family $\Gamma^0_n\subset\Gamma_n$ which comprises curves $\gamma:[0,1]\rightarrow\Aa$ given by 
\begin{equation*}
\gamma(s)=\left(a,1-\left(1-\frac{1}{n}\right)s,t\right),\; x\in[-1,1],\,t\in[-1,1].    
\end{equation*}
It is straightforward to check that the curves in $\Gamma^0_n$ are horizontal with $\gamma(0)\in E$ and $\gamma(1)\in F_n$ for all $n\in\N$, $n\ge 2$.
Further, from \eqref{horiz speed} we obtain that 
\begin{equation*}
    |\dot\gamma(s)|_H=\frac{1-\frac{1}{n}}{2\left(1-\left(1-\frac{1}{n}\right)s\right)}.
\end{equation*}
If now $\rho\in{\rm Adm}(\Gamma^0_n)$,  then we have
  \begin{equation*}
      \int_0^1 \rho\left(a,1-\left(1-\frac{1}{n}\right)s,t\right)\frac{1-\frac{1}{n}}{2\left(1-\left(1-\frac{1}{n}\right)s\right)}ds\geq1,
  \end{equation*}
which, under integration by substitution with $\lambda(s)=1-\left(1-\frac{1}{n}\right)s$, gives
\begin{equation}\label{integ1}
    \int_\frac{1}{n}^1\frac{\rho(a,\lambda,t)}{2\lambda}\,d\lambda\geq1,\,\forall n\geq2.
\end{equation}
Next, by integrating \eqref{integ1} with respect to  $x\in[-1,1]$ and $t\in[-1,1]$, we obtain
\begin{equation}\label{integ2}
    \int_{-1}^1\int_{-1}^1\int_{\frac{1}{n}}^1\frac{\rho(a,\lambda,t)}{\lambda}\,da\,d\lambda\,dt\geq8,\,\forall n\geq2.
\end{equation}
At this point, for $n\ge 2$ we define the sets $$P_n=\left\{(a,\lambda,t)\in\Aa:\;a\in[-1,1],\lambda\in\left[\frac{1}{n},1\right],\;t\in[-1,1]\right\}$$ and we apply H\"older's inequality in \eqref{integ2} with respect to $\frac{\rho(a,\lambda,t)}{\sqrt{\lambda}}\cdot\frac{\mathcal{X}_{P_n}(a,\lambda,t)}{\sqrt{\lambda}}$ and with conjugated exponents $4$ and $\frac{4}{3}$, to obtain
\begin{equation}\label{ineq depending on n}
    \left(\int_{\Aa}\rho^4(a,\lambda,t)\frac{da\,d\lambda\,dt}{\lambda^2}\right)^{\frac{1}{4}}\left(\int_{\Aa}\mathcal{X}_{P_n}(a,\lambda,t)\frac{1}{\lambda^\frac{2}{3}}\,da\,d\lambda\,dt\right)^{\frac{3}{4}}\geq8,\quad\forall n\geq2.
\end{equation}
Now we observe that
\begin{equation*}
    \int_{\Aa}\mathcal{X}_{P_n}(a,\lambda,t)\frac{1}{\lambda^\frac{2}{3}}\,da\,d\lambda\,dt\le 4\int_0^1\frac{1}{\lambda^\frac{2}{3}}\,d\lambda=12,\quad \forall n\geq2.
\end{equation*}
The latter inequality combined with \eqref{ineq depending on n} gives
\begin{equation*}
    \int_{\Aa}\rho^4(a,\lambda,t)\frac{da\,d\lambda\,dt}{\lambda^2}\geq\frac{2^4}{3^3}.
\end{equation*}
Finally, the proof is concluded by taking the infimum over all $\rho\in {\rm Adm}(\Gamma^0_n)$.
\end{proof}

The first consequence of Theorem \ref{main} is the following: 
\begin{cor}
    There is no QC map between $\Aa$ and $\H$ or $\Aa$ and $\RT$. 
\end{cor}
For a map $f$ between metric spaces, let's recall the quantity $H_f(\cdot)$ from \eqref{metricKqc} and define $B_f$ as the branch set (i.e., the set of points where $f$ does not define a local homeomorphism). We use the following definition of quasiregular (QR) maps from \cite{FLP} 
\begin{defn}
    Let $M$ and $N$ be any sub-Riemannian manifolds among $\H$, $\RT$ and $\Aa$. We
call a mapping $f : M \to N$ $K$-quasiregular if it is constant, or if:
\begin{enumerate}[(1)]
    \item $f$ is a branched cover onto its image
(i.e., continuous, discrete, open and sense-preserving),
    \item $H_f(\cdot)$ is locally bounded on $M$,
    \item $H_f(p) \leq K$ for almost every $p \in M$,
    \item  the branch set $B_f$ and its image have measure zero.
\end{enumerate}
A mapping is said to be \textit{quasiregular} if it is $K$-quasiregular for some $1\leq K<\infty$. 
\end{defn} 

From the definition it is clear, every QC map is QR. On the other hand, the class of QR maps can be substantially larger than the class of QC maps.

Let us recall, that by Theorem 4.8.1 from \cite{FLT} if $f: \H \to N$ is a QR map where $N$ is $4$-hyperbolic then $f$ must be constant. Applying this statement to our situation, we obtain the following result:

 \begin{thm} If $f: \H \to \Aa$ is a quasiregular map, then $f$ is constant. 
     \end{thm}
In contrast to the previous statement, we note that there can be plenty of examples of QR maps $f:\Aa\to\H$. One such map is the following:
\begin{exam}
    Let $f:\Aa\to\H$ be the map defined by
\begin{equation*}
f(a,\lambda,t)=(-\sqrt{\lambda}\cos a,\sqrt{\lambda}\sin a,t),\;(a,\lambda,t)\in\Aa .      
\end{equation*}
By a  direct calculation one can verify the contact property of $f$, namely $f^*\vartheta_\H=2\lambda\vartheta$. Moreover, denoting $f(a,\lambda,t)=(x,y,t)\in\H$, one can check that $f_*U=yX-xY$ and $f_*V=xX+yY$. Using this, we have that 
$$f_*(\alpha U +\beta V)= (\alpha y +\beta x)X + (\beta y-\alpha x) Y,
$$
for any $\alpha, \beta \in \R$.

Since $\{U, V\}$, resp. $\{X,Y\}$, is the orthonormal basis in the sub-Riemannian metric of $\Aa$, resp. $\H$, we obtain that 
$$|f_*(\alpha U + \beta V)|_{\H} = \sqrt{(\alpha^2 + \beta^2)(x^2 + y^2)}$$
and therefore
$$ H_f(a, \lambda, t)= \frac{\max \{ |f_*(\alpha U + \beta V)|_{\H}: \alpha^2 + \beta^2 =1 \} }{\min \{ |f_*(\alpha U + \beta V)|_{\H}: \alpha^2 + \beta^2 =1 \}}=1,$$
for every point $(a, \lambda,t) \in \Aa$.  See also Proposition 2.4 in \cite{FLP} for a different way to compute the value of $H_f(\cdot)$.

Furthermore, note, that  a direct computation, gives  $\det f_*= \frac{1}{2}$; and thus $f$ is a local diffeomorphism at every point. This means that the branch set  $B_f$  of $f$ is empty, and thus $f$ is an immersion of $\Aa$ into $\H$. Consequently we conclude that $f$ is $1$-quasiregular. 

Further examples of QR maps $g: \Aa \to \H$ can be obtained as compositions  $g= h\circ f $ where $h: \H \to \H$ is a QC map of the Heisenberg group. 
\end{exam}

\bibliographystyle{plain}  
\bibliography{My_Library}    

\Addresses

\end{document}